\documentclass[12pt,a4paper]{amsart}
\usepackage{amssymb,amsmath}
\usepackage{hyperref}

\numberwithin{equation}{section}
\newtheorem{theorem}{Theorem}[section]
\newtheorem{lemma}[theorem]{Lemma}
\newtheorem{proposition}[theorem]{Proposition}
\newtheorem{corollary}[theorem]{Corollary}

\theoremstyle{definition}
\newtheorem{definition}[theorem]{Definition}

\theoremstyle{remark}
\newtheorem{remark}[theorem]{Remark}

\newcommand\Supp{\operatorname{Supp}}

\newcommand\Tor{\operatorname{Tor}}
\newcommand\Hom{\operatorname{Hom}}

\newcommand\Ext{\operatorname{Ext}}
\newcommand\Rad{\operatorname{Rad}}

\newcommand\cd{\operatorname{cd}}

\newcommand\grade{\operatorname{grade}}

\newcommand{\qism}{\stackrel{\sim}{\longrightarrow}}

\begin{document}
\title[Cohomologically complete intersections]{Cohomologically complete intersections with vanishing of Betti numbers}%
\author[W. Mahmood]{Waqas Mahmood }
\address{Quaid-I-Azam University Islamabad,Pakistan}%
\email{ waqassms$@$gmail.com}

\thanks{This research was partially supported by Higher Education Commission, Pakistan}
\subjclass[2000]{13D45.}
\keywords{Local cohomology, Cohomologically complete intersections, Betti numbers}%

\maketitle
\begin{abstract} Let $I$ be ideal of an $n$-dimensional local Gorenstein ring $R$. In this paper we will describe several necessary and sufficient conditions such that the ideal $I$ becomes cohomologically complete intersections. In fact, as a technical tool, it will be shown that the vanishing $H^i_{I}(R)= 0$ for all $i\neq c= \grade (I)$ is equivalent to the vanishing of the Betti numbers of $H^c_{I}(R)$. This gives a new characterization to check the cohomologically complete intersections property with the homological properties of the vanishing of Tor modules of $H^c_{I}(R)$.
\end{abstract}

\section{Introduction}
For a commutative Noetherian local ring $(R,\mathfrak m,k)$ and an ideal $I\subset R$ we denote $H^i_I(R)$, $i\in \mathbb{Z}$, the local cohomology modules of $R$ with respect to $I$. We refer to see \cite{b} and \cite{goth} for the definition of local cohomology modules. It is a one of the difficult questions to compute the cohomological dimension $\cd (I)$ of $I$ with respect to $R$. Here $\cd (I):=\max\{i: i\in \mathbb{Z}\}.$ Moreover it is well-known that $\grade (I)\leq \cd (I)$.

The ideal $I$ is called cohomologically complete intersections if $H^i_{I}(R)= 0$ for all $i\neq c= \grade (I)$. Note that cohomologically complete intersections property helps us to decide whether an ideal is set-theoretically complete. As a first step M. Hellus and P. Schenzel (see \cite[Theorem 0.1]{pet1}) have shown that $H^i_{I}(R)= 0$ for all $i\neq c= \grade (I)$ if and only if $\dim_k(\Ext^i_R(k,H^c_{I}(R)))= \delta_{n,i}$ provided that $I$ is cohomologically complete intersection in $V(I) \setminus \{ \mathfrak m\}$ over an $n$-dimensional local Gorenstein ring $R$. This was the first time that the cohomologically complete intersections property of $I$ is completely encoded in homological properties of the module $H^c_{I}(R).$ Moreover the above characterization of cohomologically complete intersections looks like a Gorensteiness property.

After that several authors have studied this cohomologically complete intersections property. For instance the author and M. Zargar (see \cite[Theorem 1.1]{waqas2} and \cite[Theorem 1.1]{z}) have generalized this result to a maximal Cohen-Macaulay module of finite injective dimension over a finite dimensional local ring. For an extension to arbitrary finitely generated $R$-module we refer to \cite[Theorem 4.4]{p1}.

Here a natural question arise that can it also be true in case of the vanishing of Betti numbers $\Tor^{R}_{i}(k,H^c_{I}(R))$, $i\in \mathbb{Z}$, of $H^c_{I}(R)$. In this regard we will prove the following result:
\begin{theorem}\label{a1}
Let $R$ be a local Gorenstein ring of dimension $n$ and $I$ be an ideal with $c = \grade(I).$ Then  for all $\mathfrak{p}\in V(I)$ the following conditions are equivalent:
\begin{itemize}
\item[(a)] $H^i_I(R)= 0$ for all $i\neq c,$ that is $I$ is a cohomologically complete intersection.

\item[(b)] The natural homomorphism
\[
\Tor^{R_\mathfrak{p}}_{c}(E_{R_\mathfrak{p}}(k(\mathfrak{p})),H^c_{IR_\mathfrak{p}}(R_\mathfrak{p})))\to E_{R_\mathfrak{p}}(k(\mathfrak{p}))
\]
is an isomorphism and $\Tor^{R_\mathfrak{p}}_{i}(E_{R_\mathfrak{p}}(k(\mathfrak{p})),H^c_{IR_\mathfrak{p}}(R_\mathfrak{p}))= 0$ for all $i\neq c .$
\item[(c)] The natural homomorphism
\[
 E_{R_\mathfrak{p}}(k(\mathfrak{p}))\to H^c_{\mathfrak{p}R_\mathfrak{p}}(\Hom_{R_\mathfrak{p}}(H^c_{IR_\mathfrak{p}}(R_\mathfrak{p}), E_{R_\mathfrak{p}}(k(\mathfrak{p}))))
\]
is an isomorphism and $H^i_{\mathfrak{p}R_\mathfrak{p}}(\Hom_{R_\mathfrak{p}}(H^c_{IR_\mathfrak{p}}(R_\mathfrak{p}), E_{R_\mathfrak{p}}(k(\mathfrak{p}))))= 0$ for all $i\neq c$.

\item[(d)] The natural homomorphism
\[
\Tor^{R_\mathfrak{p}}_{c}(k(\mathfrak{p}),H^c_{IR_\mathfrak{p}}(R_\mathfrak{p}))\to k(\mathfrak{p})
\]
is an isomorphism and $\Tor^{R_\mathfrak{p}}_{i}(k(\mathfrak{p}),H^c_{IR_\mathfrak{p}}(R_\mathfrak{p}))= 0$ for all $i\neq c$.
\end{itemize}
\end{theorem}
In the above Theorem \ref{a1} $k(\mathfrak{p})$ denotes the residue field of the local Gorenstein ring $R_\mathfrak{p}$ with the injective hull $E_{R_\mathfrak{p}}(k(\mathfrak{p}))$. Moreover the existence of the above natural homomorphisms is shown in Theorem \ref{2}. The new point of view here is a new characterization of an ideal $I$ to be cohomologically complete is equivalent to the following property of Betti numbers of $H^c_{IR_\mathfrak{p}}(R_\mathfrak{p})$
\[
\dim_{k(\mathfrak{p})} (\Tor^{R_\mathfrak{p}}_{i}(k(\mathfrak{p}),H^c_{IR_\mathfrak{p}}(R_\mathfrak{p})))= \delta_{c,i}
\]
for all ${\mathfrak p}\in V(I)$.

\section{Preliminaries}
In this section we will recall few preliminaries and auxiliary results. In the paper we will denote by $(R,\mathfrak{m})$ a commutative Noetherian local ring of finite dimension $n$ with unique maximal ideal $\mathfrak{m}$. Let $E= E_R(k)$ be the injective hull of the residue field $k=R/\mathfrak{m}$. We will denote $D(\cdot)$ by the Matlis dual functor. Moreover for the basic facts about commutative algebra and homological algebra we refer to \cite{b}, \cite{har1}, \cite{m}, and \cite{w}.

Suppose that the homomorphism $X\to Y$ of complexes of $R$-modules induces an isomorphism in homologies. Then it is called quasi-isomorphism. In this case we will write it as $X \qism Y$.

In order to derive the natural homomorphisms of the next Theorem \ref{2} we need the definition of the truncation complex. The truncation complex is firstly introduce in \cite[Definition 4.1]{p2}. Let $E^{\cdot}_R(R)$ be a minimal injective resolution of a local Gorenstein ring $R$ of $\dim(R)= n$. Let $I$ be an ideal of $R$ with $\grade (I)= c$. Then there is an exact sequence
\[
0\to H^c_I(R)\to \Gamma_I(E^{\cdot}_R(R))^c\rightarrow \Gamma_I(E^{\cdot}_R(R))^{c+1}.
\]
Whence there is an embedding of complexes of $R$-modules $ H^c_I(R)[-c]\rightarrow \Gamma_I(E^{\cdot}_R(R))$.

\begin{definition} \label{2.2}
Let $C^{\cdot}_R(I)$ be the cokernel of the above embedding. Then the
truncation complex of $R$ with respect to $I$ is a short exact
sequence of complexes
\[
0\to H^c_I(R)[-c]\to \Gamma_I(E^{\cdot}_R(R))\to C^{\cdot}_R(I)\to 0.
\]
Note that one can easily see from the long exact sequence of cohomologies that $H^i(C^{\cdot}_R(I))= 0$ for all $i\leq c$ or $i> n$ and $H^i(C^{\cdot}_R(I))\cong H^i_I(R)$ for all $c< i\leq n.$
\end{definition}

As a consequence of the truncation complex the following result was proved in \cite[Lemma 2.2 and Corollary 2.3]{pet1}. Note that a generalization of the next result to maximal Cohen-Macaulay modules was given in \cite[Theorem 4.2 and Corollary 4.3]{waqas2}. Moreover recently it has also been extended to finitely generated $R$-modules (see \cite[Lemma 4.2]{p1}). For the sake of completeness we add it here.

\begin{lemma} \label{1}
Let $(R,\mathfrak{m})$ denote an $n$-dimensional Gorenstein ring. Let $I\subset R$ denote an ideal with $c = \grade I$. Then there is a short exact
sequence
\[
0 \to H^{n-1}_{\mathfrak m}(C^{\cdot}_R(I)) \to H^d_{\mathfrak m}(H^c_I(R)) \to E \to H^n_{\mathfrak m}(C^{\cdot}_R(I)) \to 0,
 \]
and isomorphisms $H^{i-c}_{\mathfrak m}(H^c_I(R)) \cong H^{i-1}_{\mathfrak m}(C^{\cdot}_R(I))$ for all $i\neq n,n+1$. Moreover if in addition $H^i_I(R)=0$ for all $i\neq c$ then the map $H^d_{\mathfrak m}(H^c_I(R)) \to E$ is an isomorphism and $H^{i}_{\mathfrak m}(H^c_I(R))=0$ for all $i\neq d.$
\end{lemma}

\begin{proof}
For the proof see \cite[Lemma 2.2 and Corollary 2.3]{pet1}.
\end{proof}

In the following we will obtain some natural homomorphisms with the help of the truncation complex. These maps will be used further to investigate the property of cohomologically complete intersection ideals.
\begin{theorem}\label{2}
Let $R$ be a Gorenstein ring of dimension $n$ and $I$ an ideal with $c = \grade(I).$ Then we have the following results:
\begin{itemize}
\item[(a)]  There is an exact sequence
\[
0\to \Tor_{-1}^R(E,C^{\cdot}_R(I))\to \Tor_{c}^R(E,H^c_I(R))\to E\to \Tor_{0}^R(E,C^{\cdot}_R(I))\to 0
\]
and isomorphism $\Tor_{c-i}^R(E, H^c_I(R))\cong \Tor_{-(i+1)}^R(E,C^{\cdot}_R(I))$ for all $i\neq 0, -1$
\item[(b)] There is an exact sequence
\[
0\to H^0_{\mathfrak m}(D(C^{\cdot}_R(I)))\to E\to H^c_{\mathfrak m}(D(H^c_I(R)))\to H^1_{\mathfrak m}(D(C^{\cdot}_R(I)))\to 0
\]
and isomorphism $H^{c+i}_{\mathfrak m}(D(H^c_I(R)))\cong H^{i+1}_{\mathfrak m}(D(C^{\cdot}_R(I)))$ for all $i\neq 0, -1$.
\item[(c)] There is an exact sequence
\[
0\to \Tor_{-1}^R(k,C^{\cdot}_R(I))\to \Tor_{c}^R(k,H^c_I(R))\to k\to \Tor_{0}^R(k,C^{\cdot}_R(I))\to 0
\]
and isomorphism $\Tor_{c-i}^R(k,H^c_I(R))\cong \Tor_{-(i+1)}^R(k,C^{\cdot}_R(I))$ for all $i\neq 0, -1$.
\end{itemize}
\end{theorem}

\begin{proof}
Let $F_{\cdot}^R$ be a free resolution of $E$. Then the short exact sequence of the truncation complex induces the following short exact sequence of complexes of $R$-modules
\begin{equation}\label{1aa}
0\to (F_{\cdot}^R\otimes_{R} H^c_I(R))[-c]\to
F_{\cdot}^R\otimes_{R}\Gamma_I(E^{\cdot}_R(R))\to F_{\cdot}^R\otimes_{R} C^{\cdot}_R(I)\to 0
\end{equation}
Now let $\underline{y} = y_1,\ldots, y_r$ be a generating set of the ideal $I$ and $\Check{C}_{\underline{y}}$ denote the \v{C}ech complex with respect
to $\underline{y}$. Since $E^{\cdot}_R(R)$ is a complex of injective $R$-modules. Then by \cite[Theorem 3.2]{pet2} the middle complex is a quasi-isomorphic to $F_{\cdot}^R\otimes_{R} \Check{C}_{\underline{y}}\otimes_{R} E^{\cdot}_R(R)$.

Moreover $\Check{C}_{\underline{y}}$ and $F_{\cdot}^R$ are complexes of flat $R$-modules so there are the following quasi-isomorphisms
\[
F_{\cdot}^R\otimes_{R}\Check{C}_{\underline{y}} \qism F_{\cdot}^R\otimes_{R} \Check{C}_{\underline{y}}\otimes_{R} E^{\cdot}_R(R), \text { and }
\]
\[
F_{\cdot}^R\otimes_R \Check{C}_{\underline{y}} \qism E\otimes_R \Check{C}_{\underline{y}}
\]
But $\Supp_R(E)= V(\mathfrak{m})$ it follows that $E\otimes_R \Check{C}_{\underline{y}}\cong E$. Therefore we get the homologies $H^i(F_{\cdot}^R\otimes_{R}\Gamma_I(E^{\cdot}_R(R)))=0$ for all $i\neq 0$ and $E$ for $i=0$. Then the statement $(a)$ can be deduced from the homology sequence of the above exact sequence \ref{1aa}.

Now apply the functor $D(\cdot)=\Hom_R({\cdot}, E)$ to the truncation complex then it induces the short exact sequence
\begin{equation}\label{e}
0\to D(C^{\cdot}_R(I))\to D(\Gamma_I(E^{\cdot}_R(R)))\to D(H^c_I(R))[c]\to 0.
\end{equation}
Let $F_{\cdot}(R/\mathfrak{m}^s)$ be a free resolution of $R/\mathfrak{m}^s$ for each $s\in \mathbb{N}$. Take $\Hom_R(F_{\cdot}(R/\mathfrak{m}^s),{\cdot})$ to this sequence then the resulting exact sequence of complexes is the following:
\begin{gather*}
0\to \Hom_R(F_{\cdot}(R/\mathfrak{m}^s),D(C^{\cdot}_R(I)))\to \Hom_R(F_{\cdot}(R/\mathfrak{m}^s),D(\Gamma_I(E^{\cdot}_R(R))))\to \\ \Hom_R(F_{\cdot}(R/\mathfrak{m}^s),D(H^c_I(R)))[c] \to 0.
\end{gather*}
Now we look at the cohomologies of the complex in the middle. By Hom-Tensor Duality it is isomorphic to $D(F_{\cdot}(R/\mathfrak{m}^s)\otimes_{R}\Gamma_I(E^{\cdot}_R(R)))$.
Since the Matlis dual functor $D(\cdot)$ is exact and cohomologies commutes with exact functor. So there is an isomorphism
\[
H^i(D(F_{\cdot}(R/\mathfrak{m}^s)\otimes_{R}\Gamma_I(E^{\cdot}_R(R))))\cong D(H^i(F_{\cdot}(R/\mathfrak{m}^s)\otimes_{R}\Gamma_I(E^{\cdot}_R(R))))
\]
for all $i\in \mathbb{Z}$. Then, by the same arguments as we used in the proof of $(a)$, the cohomologies on the right side are zero for each $i\neq 0$ and for $i=0$ it is isomorphic to $D(R/\mathfrak{m}^s)$. Recall that support of $R/\mathfrak{m}^s$ is contained in $V(\mathfrak{m}).$ Then by cohomology sequence there is an exact sequence
\begin{gather*}\label{we}
0\to \Hom_R(R/\mathfrak{m}^s,D(C^{\cdot}_R(I)))\to D(R/\mathfrak{m}^s)\to \Ext^c_R(R/\mathfrak{m}^s,D(H^c_I(R)))\to \\
\Ext^1_R(R/\mathfrak{m}^s,D(C^{\cdot}_R(I)))\to 0
\end{gather*}
and the isomorphism $\Ext^{i+1}_R(R/\mathfrak{m}^s,D(C^{\cdot}_R(I)))\cong \Ext^{i+c}_R(R/\mathfrak{m}^s,D(H^c_I(R)))$ for all $i\neq 0,-1$. Take the direct limit we get the following exact sequence
\[
0\to H^0_{\mathfrak m}(D(C^{\cdot}_R(I)))\to E\to H^c_{\mathfrak m}(D(H^c_I(R)))\to H^1_{\mathfrak m}(D(C^{\cdot}_R(I)))\to 0
\]
and for all $i\neq 0, -1$
\[
H^{c+i}_{\mathfrak m}(D(H^c_I(R)))\cong H^{i+1}_{\mathfrak m}(D(C^{\cdot}_R(I))).
\]
This proves the statement in $(b)$.

Now let $L^R_{\cdot}$ denote the free resolution of $k$. Then short exact sequence of truncation complex induces the exact sequence
\begin{equation}\label{e1}
0\to (L^R_{\cdot}\otimes_R H^c_I(R))[-c]\to L^R_{\cdot}\otimes_R \Gamma_I(E^{\cdot}_R(R))\to L^R_{\cdot}\otimes_R C^{\cdot}_R(I)\to 0.
\end{equation}
Then by the proof of $(a)$ the cohomologies of the middle complex of this last sequence are $H^i(L^R_{\cdot}\otimes_R \Gamma_I(E^{\cdot}_R(R)))=0$ for all $i\neq 0$ and for $i=0$ it is $k$. This gives the statement in $(c)$ by virtue of long exact sequence of cohomologies. This finishes the proof of the Theorem.
\end{proof}

We close this section with the the following version of the Local Duality Lemma. The proof of it can be found in \cite{goth}. Note that in \cite[Theorem 6.4.1]{he}, \cite[Theorem 3.1]{p1}, or \cite[Lemma 2.4]{waqas1} there is a generalization of it to arbitrary cohomologically complete intersection ideals.
\begin{lemma} \label{2.1}
Let $I$ be an ideal of a Gorenstein ring $R$ with $\dim(R)= n$. Then for any $R$-module $M$ and for all $i\in \mathbb{Z}$ we have
\begin{itemize}
\item[(1)] $\Tor_{n-i}^R(M, E) \cong H^i_{\mathfrak{m}}(M)$.
\item[(2)] $D(H^i_\mathfrak{m}(M)) \cong \Ext^{n-i}_R(M, \hat{R})$. Here $\hat{R}$ denotes the of completion of $R$ with respect to the maximal ideal.
\end{itemize}
\end{lemma}

\section{On cohomologically complete intersections}
In this section we will see when the natural homomorphisms, described in above Theorem \ref{2}, are isomorphisms. The first result in this regard is obtained in the following Corollary provided that the ideal $I$ is of cohomologically complete intersections. In fact the next result tells us that all the Betti numbers of the module $H^c_I(R)$ vanish except at degree $c$.
 \begin{corollary}\label{01}
With the notation of Theorem \ref{2} suppose in addition that $H^i_I(R)= 0$ for all $i\neq c.$ Then the following are true:
\begin{itemize}
\item[(a)] The natural homomorphism
\[
\Tor_{c}^R(E,H^c_I(R))\to E
\]
is an isomorphism and $\Tor_{i}^R(E,H^c_I(R))=0$ for all $i\neq c$
\item[(b)] The natural homomorphism
\[
E\to H^c_{\mathfrak m}(D(H^c_I(R)))
\]
is an isomorphism and $H^{i}_{\mathfrak m}(D(H^c_I(R)))=0$ for all $i\neq c$.
\item[(c)] The natural homomorphism
\[
\Tor_{c}^R(k,H^c_I(R))\to k
\]
is an isomorphism and $\Tor_{i}^R(k,H^c_I(R))=0$ for all $i\neq c$. That is the Betti numbers of $H^c_I(R)$ satisfy
\[
\dim_k(\Tor^{R}_i(k,H^c_I(R)))=\delta_{c,i}.
\]
\end{itemize}
\end{corollary}

\begin{proof}
Since $H^i_I(R)= 0$ for all $i\neq c$ so the complex $C^{\cdot}_R(I)$ is an exact complex (by definition of the truncation complex). Let $F_{\cdot}^R$, $F_{\cdot}(R/\mathfrak{m}^s)$ and $L^{\cdot}_R$ be the free resolutions of $E$, $R/\mathfrak{m}^s$ and $k$ respectively. Then it follows that all the following complexes
\[
F_{\cdot}^R\otimes_R C^{\cdot}_R(I),
\]
\[
\Hom_R(F_{\cdot}(R/\mathfrak{m}^s),D(C^{\cdot}_R(I))),\text { and }
\]
\[
L^{\cdot}_R\otimes_R C^{\cdot}_R(I)
\]
are exact. Then from the long exact sequence of cohomologies of the sequences \ref{1aa} and \ref{e1} we can easily obtained the statements in $(a)$ and $(c)$. Moreover after applying the functor $\Hom_R(F_{\cdot}(R/\mathfrak{m}^s),D(C^{\cdot}_R(I)))$ to the sequence \ref{e} we get, from the cohomology sequence, that the natural homomorphism
\[
D(R/\mathfrak{m}^s)\to \Ext^c_R(R/\mathfrak{m}^s,D(H^c_I(R)))
\]
is an isomorphism and $\Ext^{i}_R(R/\mathfrak{m}^s,D(H^c_I(R)))=0$ for all $i\neq c$. By passing to the direct limits we get the statement in $(b)$. Hence the proof of the Corollary is complete.
\end{proof}

In order to prove Theorem \ref{a} which is one of the main result of this section we need the following result. Note that the next Theorem provides us the equivalence of the natural homomorphisms of Theorem \ref{2} such that they become isomorphisms. Moreover this equivalence related to the vanishing of the Betti numbers of $H^c_I(R)$ for $\grade(I)=c$.

\begin{theorem}\label{6}
Let $I$ be an ideal of an $n$-dimensional Gorenstein ring $R$ with $\grade(I)=c$. Then the following conditions are equivalent:
\begin{itemize}
\item[(a)] The natural homomorphism
\[
E\to H^c_{\mathfrak m}(D(H^c_I(R)))
\]
is an isomorphism and $H^i_{\mathfrak m}(D(H^c_I(R)))=0$ for all $i\neq c.$
\item[(b)] The natural homomorphism
\[
\Tor_{c}^R(k,H^c_I(R))\to k
\]
is an isomorphism and $\Tor_{i}^R(k,H^c_I(R))= 0$ for all $i\neq c$.
\item[(c)] The Betti numbers of $H^c_I(R)$ satisfy
\[
\dim_k(\Tor^{R}_i(k,H^c_I(R)))=\delta_{i,c}.
\]
\end{itemize}
\end{theorem}

\begin{proof}
Note that the equivalence of $(b)$ and $(c)$ is obvious. Now we prove that $(a)$ implies $(b)$. Since $H^i_{\mathfrak m} (D(H^c_I(R)))= 0$ for all $i\neq c$. By Theorem \ref{2} $(a)$ it follow that $H^i_{\mathfrak m} (D(C^{\cdot}_R(I)))= 0$ for all $i\in \mathbb{Z}$. Let $\Check{C}_{\underline{x}}$ be the \v{C}ech complex with respect to $\underline{x}= x_1, \ldots ,x_s\in \mathfrak{m}$ such that $\Rad \mathfrak{m}= \Rad(\underline{x})R$. Then it implies that $\Check{C}_{\underline{x}}\otimes_{R} D(C^{\cdot}_R(I))$ is an exact complex.

Suppose that $F^{\cdot}_R$ denote a minimal injective resolution of $D(C^{\cdot}_R(I)).$ Let us denote $X:= \Hom_R(k, F^{\cdot}_R)$ then there is an isomorphism
\[
\Ext^{i}_R(k, D(C^{\cdot}_R(I)))\cong H^i(X)
\]
for all $i\in \mathbb{Z}$. We claim that the complex $X$ is homologically trivial. To this end note that the support of each module of $X$ is in $\{\mathfrak{m}\}$. It follows that there is an isomorphism of complexes
\[
\Check{C}_{\underline{x}}\otimes_{R} X\cong X.
\]
So in order to prove the claim it will be enough to show that $H^i(\Check{C}_{\underline{x}}\otimes_{R} X)=0$ for all $i\in \mathbb{Z}$. By the above arguments, for a free resolution $L_{\cdot}^R$ of $k$, the following complex is exact
\[
Y:=\Hom_R(L_{\cdot}^R, \Check{C}_{\underline{x}}\otimes_{R} D(C^{\cdot}_R(I))).
\]
Moreover $L_{\cdot}^R$ is a right bounded complex of finitely generated free $R$-modules and $\Check{C}_{\underline{x}}$ is a bounded complex of flat $R$-modules. So by \cite[Proposition 5.14]{har1} $Y$ is quasi-isomorphic to $\Check{C}_{\underline{x}}\otimes_{R} \Hom_R(L_{\cdot}^R, D(C^{\cdot}_R(I)))$. So it is homologically trivial. Note that the morphism of complexes
\[
\Check{C}_{\underline{x}}\otimes_{R} \Hom_R(L_{\cdot}^R, D(C^{\cdot}_R(I)))\to \Check{C}_{\underline{x}}\otimes_{R} \Hom_R(L_{\cdot}^R, F^{\cdot}_R)
\]
induces an isomorphism in cohomologies. Because of $F^{\cdot}_R$ is a minimal injective resolution of $C^{\cdot}_M(I)$. Moreover
\[
\Check{C}_{\underline{x}}\otimes_{R} X \qism \Check{C}_{\underline{x}}\otimes_{R} \Hom_R(L_{\cdot}^R, F^{\cdot}_R).
\]
By the discussion above the complex on the right side is homologically trivial. It follows that the complex $\Check{C}_{\underline{x}}\otimes_{R} X$ is homologically trivial. This proves the claim. Therefore $\Ext^{i}_R(k, D(C^{\cdot}_R(I)))=0$ for all $i\in \mathbb{Z}$.

By Hom-Tensor Duality $D(\Tor_{i}^R(k,C^{\cdot}_R(I)))\cong\Ext^{i}_R(k, D(C^{\cdot}_R(I)))=0$ for all $i\in \mathbb{Z}$. It implies that $L_{\cdot}^R\otimes_R C^{\cdot}_R(I)$ is an exact complex. Then from the cohomology sequence of the exact sequence \ref{e1} it follows that the natural homomorphism
\[
\Tor_{c}^R(k,H^c_I(R))\to k
\]
is an isomorphism and $\Tor_{i}^R(k,H^c_I(R))= 0$ for all $i\neq c$.

Conversely, note that after the application of the functor $\Tor^{R}_i(\cdot, C^{\cdot}_R(I))$ to the exact sequence
\[
0\to \mathfrak m^s/\mathfrak m^{s+1}\to R/\mathfrak m^{s+1}\to R/\mathfrak m^{s}\to 0
\]
induces the exact sequence
\[
\Tor^{R}_i(\mathfrak m^s/\mathfrak m^{s+1}, C^{\cdot}_R(I)) \to \Tor^{R}_i(R/\mathfrak m^{s+1}, C^{\cdot}_R(I))\to \Tor^{R}_i(R/\mathfrak m^s, C^{\cdot}_R(I))
\]
for all $i\in \mathbb Z$. Then by induction on $s$ , in view of vanishing of Tot modules, this proves that $\Tor^{R}_i(R/\mathfrak m^s, C^{\cdot}_M(I)))=0$ for all $i\in \mathbb{Z}$ and for all $s\in \mathbb N$ (see Theorem \ref{2}$(c)$). It follows that
\[
\Ext^{i}_R(R/\mathfrak m^s, D(C^{\cdot}_R(I)))\cong D(\Tor_{i}^R(R/\mathfrak m^s,C^{\cdot}_R(I)))=0
\]
for all $i\in \mathbb{Z}$ and for all $s\in \mathbb N$. Recall that $\mathfrak m^s/\mathfrak m^{s+1}$ is a finite dimensional $k$-vector space. Then by virtue of the proof of Theorem \ref{2}$(b)$ the natural homomorphism 
\[
D(R/\mathfrak{m}^s)\to \Ext^c_R(R/\mathfrak{m}^s,D(H^c_I(R)))
 \]
is an isomorphism and $\Ext^{i}_R(R/\mathfrak{m}^s,D(H^c_I(R)))=0$ for all $i\neq c$. By Passing to the direct limit of it we can easily obtain the statement in $(a)$. This completes the proof of the Theorem.
\end{proof}

\begin{lemma}\label{03}
With the above notation suppose that $N:= D(D(H^c_I(R)))$ then we have:
\begin{itemize}
\item[(a)] The natural homomorphism
\[
E\to H^c_{\mathfrak m}(D(H^c_I(R)))
\]
is an isomorphism if and only if the natural homomorphism
\[
\lim_{\longleftarrow} \Tor^R_{c}(R/\mathfrak{m}^s,N)\to \hat{R}
\]
is an isomorphism.
\item[(b)] Let $i\in \mathbb{Z}$ be a fixed integer. Then the module $H^i_{\mathfrak m}(D(H^c_I(R)))= 0$ if and only if $\lim\limits_{\longleftarrow} \Tor^R_{i}(R/\mathfrak{m}^s,N)= 0$.
\end{itemize}
\end{lemma}

\begin{proof}
Since Hom functor transforms the direct system into an inverse system at first variable. Moreover $H^i_{\mathfrak{m}}(D(H^c_I(R)))\cong \lim\limits_{\longrightarrow}\Ext_R^{i}(R/\mathfrak{m}^s,D(H^c_I(R)))$ for each $i\in \mathbb{Z}$. By taking the Matlis dual of this induces the isomorphism
\[
D(H^i_{\mathfrak{m}}(D(H^c_I(R))))\cong \lim\limits_{\longleftarrow} D(\Ext_R^{i}(R/\mathfrak{m}^s,D(H^c_I(R))))
\]
for each $i\in \mathbb{Z}$. Then by Hom-Tensor duality the module on the right side is isomorphic to
\[
\lim_{\longleftarrow} \Tor^R_{i}(R/\mathfrak{m}^s,N).
\]
Then both of the statements are obvious in view of Matlis duality.
\end{proof}

The next Proposition is indeed in the proof of Theorem \ref{a}. So we will prove it firstly.
\begin{proposition}\label{7}
Let $R$ be an $n$-dimensional local Gorenstein ring. Let $I$ be an ideal with $\grade(I)=c$ and $d:=n-c$. If the natural homomorphism
\[
\Tor_{c}^R(k,H^c_I(R))\to k
\]
is an isomorphism and $\Tor_i^R(k,H^c_I(R))=0$ for all $i\neq c$. Then the following conditions hold:
\begin{itemize}
\item[(a)] The natural homomorphism
\[
H^d_{\mathfrak m}(H^c_I(R)) \to E
\]
is an isomorphism and  $H^{i}_{\mathfrak m}(H^c_I(R))=0$ for all $i\neq d$.

\item[(b)] The natural homomorphism
\[
\Tor_{c}^R(E,H^c_I(R))\to E
\]
is an isomorphism and $\Tor_{i}^R(E,H^c_I(R))=0$ for all $i\neq c$.
\end{itemize}
\end{proposition}

\begin{proof}
Note that by Local Duality Lemma \ref{2.1} (for $M=H^c_I(R)$) it will be enough to prove the statements in $(b)$. 

For this let $N_{\alpha}:=D(R/\mathfrak m^\alpha)$ for each $\alpha\in \mathbb N$. Suppose that $F_{\cdot}(N_{\alpha})$ denote a minimal free resolution of $N_{\alpha}$ for each $\alpha\in \mathbb N$. Since the support of $N_{\alpha}$ is contained in $\{\mathfrak{m}\}$. Then after the implement of the functor $ \cdot \otimes_R F_{\cdot}(N_{\alpha})$ to the truncation complex gives us the following natural homomorphisms
\[
f_\alpha: \Tor^{R}_c(N_{\alpha}, H^c_I(R))\to N_{\alpha}
\]
for all $\alpha\in \mathbb N$ (by the proof of Theorem \ref{2}$(a)$). Now the short exact sequence $0\to \mathfrak m^\alpha/\mathfrak m^{\alpha+1}\to R/\mathfrak m^{\alpha+1}\to R/\mathfrak m^{\alpha}\to 0$ induces the exact sequence
\[
0\to N_{\alpha}\to N_{\alpha+1}\to D(\mathfrak m^\alpha/\mathfrak m^{\alpha+1})\to 0
\]
Moreover after the application of the functor $\Tor^{R}_i(\cdot, H^c_I(R))$ to the this sequence we get the exact sequence
\[
\Tor^{R}_i(N_{\alpha}, H^c_I(R)) \to \Tor^{R}_i(N_{\alpha+1}, H^c_I(R))\to \Tor^{R}_i(D(\mathfrak m^\alpha/\mathfrak m^{\alpha+1}), H^c_I(R))
\]
for all $i\in \mathbb Z$. Since $N_{1}=D(R/\mathfrak m)\cong k$. Then by induction on $\alpha$ , in view of vanishing of Tor modules, this proves that $\Tor^{R}_i(N_{\alpha}, H^c_I(R)))=0$ for all $i\neq c$ and for all $\alpha\in \mathbb N$.

Now we show that $f_{\alpha}$ is an isomorphism for all $\alpha\in \mathbb N$. Clearly $f_1$ is isomorphism (by assumption). Then there is the following commutative diagram with exact rows
\[
\begin{array}{cccccccc}
  & & \Tor^{R}_c(N_{\alpha}, H^c_I(M)) & \to & \Tor^{R}_c(N_{\alpha+1}, H^c_I(M)) & \to & \Tor^{R}_c(D(\mathfrak m^\alpha/\mathfrak m^{\alpha+1}), H^c_I(M)) & \to 0\\
    &   & \downarrow {f_\alpha } &  & \downarrow {f_{\alpha+1}} &   & \downarrow f &\\
 0 & \to & N_{\alpha} & \to & N_{\alpha+1} & \to & D(\mathfrak m^\alpha/\mathfrak m^{\alpha+1}) & \to 0
\end{array}
\]
Note that the above row is exact because of $\Tor^{R}_i(N_{\alpha}, H^c_I(R))=0$ for all $i\neq c$ and for all $\alpha\in \mathbb N$. Then the natural homomorphism $f$ is an isomorphism because of $\mathfrak m^\alpha/\mathfrak m^{\alpha+1}$ is a finite dimensional $k$-vector space. Hence by induction, in view of snake lemma, it implies that $f_{\alpha}$ is an isomorphism for all $\alpha \in \mathbb{N}$. Take the direct limits of $f_{\alpha}$ it induces the isomorphism
\[
\Tor_{c}^R(E,H^c_I(R))\to E
\]
Since the direct limits commutes with Tor functor and $\Supp_R(E)\subseteq V(\mathfrak{m})$. Moreover note that the vanishing of the above Tor modules implies that $\Tor_{i}^R(E,H^c_I(R))=0$ for all $i\neq c$.
\end{proof}

Now we are able to prove our main result. Before proving it we will make some notation. Let $R$ be a Gorenstein ring of dimension $n$ and $I$ an ideal with $c = \grade(I).$ We will denote $k(\mathfrak{p})$ by the residue field of $R_\mathfrak{p}$ with injective hull $E_{R_\mathfrak{p}}(k(\mathfrak{p}))$ for $\mathfrak{p}\in V(I)$. Moreover we set $h({\mathfrak p})=\dim(R_{\mathfrak p})- c$.

\begin{theorem}\label{a}
Fix the previous notation. Then for all $\mathfrak{p}\in V(I)$ the following conditions are equivalent:
\begin{itemize}
\item[(a)] $H^i_I(R)= 0$ for all $i\neq c,$ that is $I$ is a cohomologically complete intersection.

\item[(b)] The natural homomorphism
\[
H^{h({\mathfrak p})}_{{\mathfrak p}R_{\mathfrak p}}(H^c_{IR_{\mathfrak p}}(R_{\mathfrak p})) \to E_{R_\mathfrak{p}}(k(\mathfrak{p}))
\]
is an isomorphism and $H^i_{{\mathfrak p}R_{\mathfrak p}}(H^c_{IR_{\mathfrak p}}(R_{\mathfrak p}))= 0$ for all $i\neq h({\mathfrak p})$.
\item[(c)] The natural homomorphism
\[
\Tor^{R_\mathfrak{p}}_{c}( E_{R_\mathfrak{p}}(k(\mathfrak{p})),H^c_{IR_\mathfrak{p}}(R_\mathfrak{p})))\to E_{R_\mathfrak{p}}(k(\mathfrak{p}))
\]
is an isomorphism and $\Tor^{R_\mathfrak{p}}_{i}(E_{R_\mathfrak{p}}(k(\mathfrak{p})),H^c_{IR_\mathfrak{p}}(R_\mathfrak{p}))= 0$ for all $i\neq c .$
\item[(d)] The natural homomorphism
\[
 E_{R_\mathfrak{p}}(k(\mathfrak{p}))\to H^c_{\mathfrak{p}R_\mathfrak{p}}(\Hom_{R_\mathfrak{p}}(H^c_{IR_\mathfrak{p}}(R_\mathfrak{p}), E_{R_\mathfrak{p}}(k(\mathfrak{p}))
\]
is an isomorphism and $H^i_{\mathfrak{p}R_\mathfrak{p}}(\Hom_{R_\mathfrak{p}}(H^c_{IR_\mathfrak{p}}(R_\mathfrak{p}),  E_{R_\mathfrak{p}}(k(\mathfrak{p}))))= 0$ for all $i\neq c$.

\item[(e)] The natural homomorphism
\[
\Tor^{R_\mathfrak{p}}_{c}(k(\mathfrak{p}),H^c_{IR_\mathfrak{p}}(R_\mathfrak{p}))\to k(\mathfrak{p})
\]
is an isomorphism and $\Tor^{R_\mathfrak{p}}_{i}(k(\mathfrak{p}),H^c_{IR_\mathfrak{p}}(R_\mathfrak{p}))= 0$ for all $i\neq c$.
\item[(f)] The Betti numbers of $H^c_{IR_\mathfrak{p}}(R_\mathfrak{p})$ satisfy
\[
\dim_{k(\mathfrak{p})} (\Tor^{R_\mathfrak{p}}_{i}(k(\mathfrak{p}),H^c_{IR_\mathfrak{p}}(R_\mathfrak{p})))= \delta_{c,i}
\]
\end{itemize}
\end{theorem}

\begin{proof}
We firstly prove that the statements in $(a)$ and $(b)$ are equivalent. Suppose that $H^i_I(R)= 0$ for all $i\neq c.$ Then by \cite[Proposition 2.7]{waqas2} it follows that $H^i_{{\mathfrak p}R_{\mathfrak p}}(H^c_{IR_{\mathfrak p}}(R_{\mathfrak p}))= 0$ for all $i\neq c=\grade(IR_{\mathfrak p})$ and for all $\mathfrak{p}\in V(I)$. So the result follows from Lemma \ref{1}.

Conversely, suppose that the statement in $(b)$ is true. We use induction on $\dim_R(R/IR)$. Let $\dim_R(R/IR)= 0$ then it follows that $\Rad (IR)= \mathfrak{m}$. This proves the result since $R$ is Gorenstein. Now let us assume that $\dim(R/IR)> 0$. Then it is easy to see that $\dim(R_{\mathfrak p}/IR_{\mathfrak p})< \dim(R/IR)$ for all ${\mathfrak p} \in V(I)\setminus \{ \mathfrak m\}$. 

Moreover by the induction hypothesis for all $i\neq c$ and for all ${\mathfrak p} \in V(I)\setminus \{ \mathfrak m\}$ we have
\[
H^i_{IR_{\mathfrak p}}(R_{\mathfrak p})= 0.
\]
That is $\Supp(H^i_I(R))\subseteq V({\mathfrak m})$ for all $i\neq c$. Since our assumption is true for $\mathfrak p= \mathfrak m$. Then by Lemma \ref{1} it follows that $H^i_\mathfrak m(C^{\cdot}_R(I))= 0$ for all $i\in \mathbb Z.$

Since $\Supp_R(H^i(C^{\cdot}_M(I)))\subseteq V(\mathfrak m)$. So by \cite[Lemma 2.5]{waqas2} in view of definition of the truncation complex we have
\[
0=H^i_\mathfrak m(C^{\cdot}_R(I))\cong H^i(C^{\cdot}_R(I))\cong H^i_I(R)
\]
for all $c<i\leq n$. That is $H^i_{I}(R)= 0$ for all $i\neq c$. This completes the proof of $(a)$ is equivalent to $(b)$.

Now we prove that the statements in $(b)$ and $(c)$ are equivalent. If the statement in $(b)$ is true. Then by the remarks above it follows that $ c=\grade(IR_{\mathfrak p})$ for all $\mathfrak{p}\in V(I)$. Hence by Local Duality the statement in $(c)$ is obvious. Note that the converse is also true by the similar arguments of induction on $\dim_R(R/IR)$ as we used above in view of Local Duality.

Let us prove that $(c)$ is equivalent to $(e)$. Let the statement in $(c)$ be true. Then it implies that $I$ is cohomologically complete intersections. By \cite[Proposition 2.7]{waqas2} it follows that $H^i_{{\mathfrak p}R_{\mathfrak p}}(H^c_{IR_{\mathfrak p}}(R_{\mathfrak p}))= 0$ for all $i\neq c=\grade(IR_{\mathfrak p})$ and for all $\mathfrak{p}\in V(I)$. So the result is obvious by virtue of Corollary \ref{01}.

Now suppose that the statement in $(e)$ holds. By Local Duality and Proposition \ref{7} one can prove it by following the same steps of induction on $\dim_R(R/IR)$.

In the similar way we can easily see that $(c)$ and $(d)$ are equivalent by virtue of Theorem \ref{6}, Proposition \ref{7} and Local Duality. Note that the equivalence of the statements in $(e)$ and $(f)$ is obvious. This completes the equivalence of all the statements of the Theorem.
\end{proof}

\begin{remark}
Note that in the above Theorem \ref{a} one should have localization with respect to all $\mathfrak{p}\in V(I)$. To prove this let $R = k[|x_0,x_1,x_2,x_3,x_4|]$ denote the formal power series ring over any filed $k.$ Suppose that $I =(x_0,x_1)\cap (x_1,x_2)\cap (x_2,x_3)\cap (x_3,x_4).$ Then Hellus and Schenzel (see \cite[Example 4.1]{pet1}) proved that $\dim_k \Ext^i_R(k, H^2_I(R)) = \delta_{3,i}$ and $H^i_I(R) \not= 0$ for all $i \not= 2,3$ with $\grade(I)=2$. 

Then by \cite[Corollary 4.7]{waqas2} it follows that the natural homomorphism $H^d_{\mathfrak m}(H^c_I(R)) \to E$ is an isomorphism and $H^{i}_{\mathfrak m}(H^c_I(R))=0$ for all $i\neq 3.$ This shows that the local conditions are necessary in order to  prove the Theorem \ref{a}.
\end{remark}

\end{document}